\documentclass{krm}
\usepackage{amsmath}
  \usepackage{paralist}
  \usepackage{graphics}
  \usepackage{epsfig}
  \usepackage{graphicx}
 \usepackage[colorlinks=true]{hyperref}
\hypersetup{urlcolor=blue, citecolor=red}

\newtheorem{theorem}{Theorem}[section]
\newtheorem{corollary}{Corollary}

\newtheorem{lemma}[theorem]{Lemma}

\theoremstyle{definition}

\newtheorem{remark}{Remark}

\newcommand{\mS}{\mathcal{S}}
\newcommand{\noi}{\noindent}

\title[Long-time behavior of plasma in 1.5D]
      {On long-time behavior of monocharged and neutral plasma in one and one-half dimensions}

\author[Robert Glassey, Stephen Pankavich and Jack Schaeffer]{}


 \email{glassey@indiana.edu}
 \email{sdp@uta.edu}
 \email{js5m@math.cmu.edu}

\begin{document}
\maketitle

\centerline{\scshape Robert Glassey}
\medskip
{\footnotesize
 \centerline{Department of Mathematics}
   \centerline{Indiana University}
   \centerline{Bloomington, IN 47405, USA}
} 

\medskip

\centerline{\scshape Stephen Pankavich}
\medskip
{\footnotesize
 \centerline{Department of Mathematics}
   \centerline{University of Texas at Arlington}
   \centerline{Arlington, TX 76019, USA}
}

\medskip

\centerline{\scshape Jack Schaeffer}
\medskip
{\footnotesize
 \centerline{Department of Mathematical Sciences}
   \centerline{Carnegie Mellon University}
   \centerline{Pittsburgh, PA 15213, USA}
}

\bigskip

\begin{abstract}
The motion of a collisionless plasma - a high-temperature, low-density, ionized gas - is described by the Vlasov-Maxwell system. In the presence of large velocities, relativistic corrections are meaningful, and when symmetry of the particle densities is assumed this formally becomes the relativistic Vlasov-Poisson system.  These equations are considered in one space dimension and two momentum dimensions in both the monocharged (i.e., single species of ion) and neutral cases.  The behavior of solutions to these systems is studied for large times, yielding estimates on the growth of particle momenta and a lower bound, uniform-in-time, on norms of the charge density.  We also present similar results in the same dimensional settings for the classical Vlasov-Poisson system, which excludes relativistic effects.
\end{abstract}

\section{Introduction}

The fundamental equations which describe the time evolution of a collisionless plasma in the presence of large velocities are given by the relativistic Vlasov-Maxwell system (RVM).  The simplest form of these equations that retains electromagnetic effects can be obtained by posing the problem in one space dimension and two momentum dimensions, the so-called ``one and one-half'' dimensional system with a single species of ion:

\begin{equation} \tag{RVM} \label{RVM} \left. \begin{gathered}
\partial_t f + \hat{v}_1 \partial_x f + (E_1 + \hat{v}_2B)\partial_{v_1} f + (E_2 - \hat{v}_1B)\partial_{v_2} f = 0\\
\rho(t,x) = \int f(t,x,v) \ dv, \qquad j(t,x) = \int \hat{v} f  \ dv \\
E_1(t,x) = \frac{1}{2}\left ( \int_{-\infty}^x \rho(t,y) dy - \int_x^\infty \rho(t,y) dy\right )\\
\partial_t B + \partial_x E_2 =0, \qquad \partial_t E_2 + \partial_x B = - j_2\\
\end{gathered} \right \}
\end{equation}
where
$$ \hat{v} = \frac{v}{\sqrt{1 + \vert v \vert^2}}$$ is the relativistic velocity.  If one further assumes that the field components $E_2$ and $B$ are initially zero and the number density $f$ is initially even in $v_2$, then these qualities persist for all $ t > 0$, and the system reduces to the ``one and one-half'' dimensional relativistic Vlasov-Poisson system:

\begin{equation} \tag{RVP} \label{RVP} \left. \begin{gathered}
\partial_t f + \hat{v}_1 \partial_x f + E_1 \partial_{v_1} f = 0 \\
\rho(t,x) = \int f(t,x,v) \ dv\\
E_1(t,x) = \frac{1}{2}\left ( \int_{-\infty}^x \rho(t,y) dy - \int_x^\infty \rho(t,y) dy\right )\\
\end{gathered} \right \}
\end{equation}
Here, it should be noted that $f = f(t,x,v_1,v_2)$ depends on two components of momentum.  Finally, if one wishes to consider a classical system without relativistic effects, then the velocity is represented by $v \in \mathbb{R}^2$ and not $\hat{v}$.  In this way, one then obtains the ``one and one-half'' dimensional classical Vlasov-Poisson system:
\begin{equation} \tag{VP} \label{VP} \left. \begin{gathered}
\partial_t f + v_1 \partial_x f + E_1 \partial_{v_1} f = 0 \\
\rho(t,x) = \int f(t,x,v) \ dv\\
E_1(t,x) = \frac{1}{2}\left ( \int_{-\infty}^x \rho(t,y) dy - \int_x^\infty \rho(t,y) dy\right )\\
\end{gathered} \right \}
\end{equation}
Though the difference between (\ref{RVP}) and (\ref{VP}) seems quite subtle (merely a ``hat'' on $v$ is included or excluded), it is well-known that the behavior of these two systems can drastically differ.  This will be further highlighted by Theorem \ref{Thm1} which displays a significant difference in the long-time dynamics of the charge density when comparing the classical and relativistic systems.  In each of the systems above $t \geq 0$ represents time, $x \in \mathbb{R}$ is position, $v \in \mathbb{R}^2$ represents momentum, $E$ and $B$ are electric and magnetic fields, $f$ is the number density in phase space of particles, and all physical constants, including the particle mass and speed of light, have been normalized.\\

In addition to the monocharged problems (\ref{RVP}) and (\ref{VP}), we will consider the analogous neutral problems that include an arbitrary number $N$ species of ion, given by:
\begin{equation} \tag{RVPN} \label{RVPN} \left.
\begin{gathered}
\partial_t f_\alpha + \hat{v}_1 \partial_x f_\alpha + e_\alpha E_1 \partial_{v_1} f_\alpha = 0 \\
\rho(t,x) = \int \sum_\alpha e_\alpha f_\alpha (t,x,v) \ dv\\
E_1(t,x) = \frac{1}{2}\left ( \int_{-\infty}^x \rho(t,y) dy - \int_x^\infty \rho(t,y) dy\right )\\
\end{gathered} \right \}
\end{equation}
and
\begin{equation} \tag{VPN} \label{VPN} \left.
\begin{gathered}
\partial_t f_\alpha + v_1 \partial_x f_\alpha + e_\alpha E_1 \partial_{v_1} f_\alpha = 0 \\
\rho(t,x) = \int \sum_\alpha e_\alpha f_\alpha (t,x,v) \ dv\\
E_1(t,x) = \frac{1}{2}\left ( \int_{-\infty}^x \rho(t,y) dy - \int_x^\infty \rho(t,y) dy\right )\\
\end{gathered} \right \}
\end{equation}
respectively, where we have a Vlasov equation for the number density $f_\alpha$ of each species indexed by $\alpha = 1,...,N$, and $e_\alpha$ is the charge of the $\alpha$ species.  Again, the $f_\alpha$ depend upon two components of momentum $v_1$ and $v_2$, and physical constants, such as particle mass, have been normalized.  In the cases of (\ref{RVPN}) and (\ref{VPN}), we will assume the condition of neutrality
\begin{equation}
\tag{N} \label{Neutrality}
\iint \sum_\alpha e_\alpha f_{0\alpha} (x,v)  \ dv \ dx = 0.
\end{equation}
which, by conservation of charge, guarantees $$\int \rho(t,x) \ dx = \int \rho(0,x) \ dx = 0$$ for all $t \geq 0$.  Another main focus of the paper will be to demonstrate the differences in qualitative behavior which arise when one compares monocharged and neutral plasmas, such as (\ref{RVP}) with (\ref{RVPN}) or (\ref{VP}) when compared with (\ref{VPN}).  This is displayed by the differences in large time behavior of particle momenta for the systems (Theorems \ref{Thm3}, \ref{Thm4}, \ref{Thm5}, and \ref{Thm6}).\\

Over the past twenty years, considerable progress has been made regarding the well-posedness of classical solutions to the Cauchy problem for Vlasov-Maxwell and Vlasov-Poisson set in a variety of dimensions (see \cite{Glassey}).  Though the issue of smooth global existence for arbitrary data remains an open question in three dimensions for both relativistic problems, it has been proven for many lower-dimensional analogues, such as (\ref{RVM}) (see \cite{GlaSch}), as well as, for the classical Vlasov-Poisson system posed in three dimensions (\cite{LP}, \cite{Pfa}, and \cite{Sch3D}).  More specifically, it is well known that solutions of (\ref{RVP}), (\ref{RVPN}), (\ref{VP}), and (\ref{VPN}) remain smooth for all $t \geq 0$ with $f(t,\cdot, \cdot)$, or $f_\alpha(t,\cdot, \cdot)$ for multi-species problems, compactly supported for all $t \geq 0$, assuming that the data possess the same properties.  While these existence theorems have contributed greatly to the mathematical understanding of the equations, very few results concerning time asymptotic behavior have appeared in the literature.  Due to the so-called ``dilation identity'', some time decay is known for Vlasov-Poisson in the classical, three-dimensional case (\cite{GS}, \cite{IR}, \cite{Per}).  Additionally, there are time decay results for the monocharged plasma (\ref{VP}) when $f$ is independent of $v_2$ (\cite{BKR}, \cite{BFFM}, \cite{Sch}).  The only known results regarding long-time behavior for a relativistic plasma were obtained in \cite{Horst} and, more recently, \cite{GPS}.  These deal with the three-dimensional, monocharged problem with spherical symmetry and the one-dimensional, neutral problem with two species, respectively.  References \cite{DD}, \cite{Dol}, and \cite{DR} are also mentioned since they deal with time-dependent rescalings and time decay for other kinetic equations.  We cite \cite{Rein} as a general reference regarding the study of the Vlasov-Poisson system.\\

Due to the large assortment of phenomena that may be exhibited by plasma of different species, one expects the behavior to be quite complicated and depend on many factors, including the size of the total charge, the sign of the net charge, and the variety of ionic species that are involved.  For simplicity, we focus on cases in which the plasma is monocharged (i.e., composed of a single species of ion) or is composed of an arbitrary number of species and satisfies the condition of neutrality (i.e., possesses zero net charge).  Hence, the methods used in the previously mentioned articles do not apply.  The present work, then, seeks to determine information regarding the large time behavior of solutions to (\ref{RVP}), (\ref{VP}), (\ref{RVPN}), and (\ref{VPN}) under these conditions.  We remark that all of the results which follow will continue to hold in the strictly one-dimensional case, in which the number density $f$ (or $f_\alpha$ if multiple species are considered) is independent of $v_2$.  However, we keep the $v_2$ dependence throughout because solutions to these equations also satisfy (\ref{RVM}), thereby providing some information regarding time asymptotics in this case.  In addition, a previous result \cite{GPS} concerning a neutral plasma in one dimension did not readily generalize to the case in which $v_2$ is included.  Thus, we believe it necessary to present the more general results for the ``one and one-half'' dimensional problems, rather than asking the reader to believe that one-dimensional results could be easily extended to this case.  Of course, we are also interested in the long-time behavior of (\ref{RVM}), and this problem is explored in the recent paper \cite{GPSRVM}. \\

Throughout the paper, we make the assumption that the initial data $f(0,x,v) = f_0(x,v)$, or in the case of multiple species $f_\alpha(0,x,v) = f_{0\alpha}(x,v)$, are $C^1$ and compactly supported.  We first consider the relativistic, monocharged case (\ref{RVP}) and state the main result that, unlike solutions in the classical case, those which satisfy (\ref{RVP}) do not give rise to a charge density that decays in time.
\begin{theorem} \label{Thm1}
Consider (\ref{RVP}) and assume $f_0 \in C_c^1(\mathbb{R}^3)$ is not identically equal to zero, then there exists $C > 0 $ such that $$\Vert \rho(t) \Vert_p \geq C$$ for all $t \geq 0, p \in [1,\infty]$.
\end{theorem}

Since \cite{GPS} showed some time decay for solutions to (\ref{RVPN}), in the case $N=2$, $e_\alpha =(-1)^{\alpha+1}$, and $f_\alpha$ independent of $v_2$, this theorem displays the distinct difference in behavior between monocharged and neutral plasmas, even in a one-dimensional setting.   Hence, the dispersive effects which occur in the neutral case seem to be a stronger or more influential phenomena than the repulsion of the monocharged case.  Additionally, while \cite{BKR} showed that the charge density decays in sup-norm like $t^{-1}$ for the classical, monocharged system (\ref{VP}) with $f$ independent of $v_2$ (and the argument can be extended to include solutions of (\ref{VP})), we find here that the charge density does not decay in any $L^p$ norm for the relativistic, monocharged system (\ref{RVP}).  Thus, we have differing asymptotic behavior depending upon the inclusion or exclusion of relativistic velocity corrections.  This can be contrasted with Horst's discovery \cite{Horst} that solutions to both the classical and relativistic 3D Vlasov-Poisson systems satisfy the same time decay estimates under the assumption of spherical symmetry.\\

For these systems, the time asymptotic behavior of solutions will depend strongly on characteristics.  Since we are first interested in (\ref{RVP}), define the associated characteristics $X(s,t,x,v)$ and $V(s,t,x,v)$ by
\begin{equation}
\label{char} \left. \begin{array}{ccc}
& & \displaystyle \frac{dX}{ds} = \widehat{V}_1 (s) \\
& & \displaystyle \frac{dV_1}{ds} = E_1(s,X(s)) \\
& & V_2(s,t,x,v) = v_2 \\
\\
& & X(t,t,x,v) = x \\
& & V_1(t,t,x,v) = v_1
\end{array} \right \}
\end{equation}
Then, in addition to lower bounds on the charge density, we can also determine a uniform upper bound under a very general assumption.
\begin{theorem} \label{Thm2}
Consider (\ref{RVP}) and assume there is $F_0 \in L^1(\mathbb{R}^2)$
such that\break $f_0(x,v) \leq F_0(v) \ \forall x \in \mathbb{R}, v
\in \mathbb{R}^2$.  Then, $$1 \leq \frac{\partial V_1}{\partial
v_1}(s,t,x,v) \quad \mathrm{for} \ 0\leq s \leq t$$ and $$\rho(t,x)
\leq \int F_0(v) \ dv \quad \mathrm{for} \ t\geq 0, x \in
\mathbb{R}.$$
\end{theorem}

For each of these systems it is well-known that $\Vert \rho(t) \Vert_1$ is constant due to charge conservation.  Hence, using Theorems \ref{Thm1} and \ref{Thm2} and interpolation with $\Vert \rho(t) \Vert_1$, we may conclude in the case of (\ref{RVP}) that $\Vert \rho(t) \Vert_p$ is $O(1)$ for any $p \in [1,\infty]$.  Finally, we can also sharply determine the asymptotic behavior of particle momenta for large time.
\begin{theorem} \label{Thm3}
Let $f(t,x,v)$ be a solution of (\ref{RVP}) and define $$Q_1(t) = \sup \{ \vert v_1 \vert : \mbox{there are} \ x, v_2 \in \mathbb{R} \ \mbox{s.t.} \ f(t,x,v) \neq 0 \}.$$ Then, there are $T, C_1, C_2 > 0$ such that for $t \geq T$, we have $$ C_1 t \leq Q_1(t) \leq C_2t.$$
\end{theorem}

To further display differences in behavior between monocharged and neutral plasma, we consider the multi-species problem (\ref{RVPN}).  Since attractive forces among particles are now introduced, one intuitively expects that the particle momenta are asymptotically slowed in comparison to those of (\ref{RVP}).  This is accurate and demonstrated by the following theorem.

\begin{theorem} \label{Thm4}
Let $f_\alpha(t,x,v)$ satisfy (\ref{RVPN}) for $\alpha = 1,...,N$ and define $$Q_1(t) = \sup \{ \vert v_1 \vert : \mbox{there are} \ x, v_2 \in \mathbb{R} \ \mbox{s.t.} \ \sum_\alpha f_\alpha (t,x,v)  \neq 0\}.$$ Then, there is $C > 0$ such that for any $t \geq 0$ we have $$Q_1(t) \leq C\sqrt{1 + t}.$$
\end{theorem}

Next, we turn our attention to the classical, monocharged problem (\ref{VP}).  As opposed to (\ref{RVP}), time decay of the charge density is known from \cite{BKR}, so one may expect faster growth of particle momenta, as well.  However, since the electric field terms are the same, the momenta are shown here to grow at exactly the same rate as in the relativistic case.

\begin{theorem} \label{Thm5}
Let $f(t,x,v)$ be a solution of (\ref{VP}) and define $$Q_1(t) = \sup \{ \vert v_1 \vert : \mbox{there are} \ x, v_2 \in \mathbb{R} \ \mbox{s.t.} \ f(t,x,v) \neq 0 \}.$$  Then, there are $T, C_1, C_2 > 0$ such that for $t \geq T$, we have $$ C_1 t \leq Q_1(t) \leq C_2t.$$
\end{theorem}

Finally, to contrast the result for (\ref{VP}), we may show that particles travel at slower speeds in the neutral case (\ref{VPN}).  In fact, using very different techniques than the proof of Theorem \ref{Thm4}, we obtain the same bound on growth as for the momenta in (\ref{RVPN}).

\begin{theorem} \label{Thm6}
Let $f_\alpha(t,x,v)$ satisfy (\ref{VPN}) for $\alpha = 1,...,N$ and define $$Q_1(t) = \sup \{ \vert v_1 \vert : \mbox{there are} \ x, v_2 \in \mathbb{R} \ \mbox{s.t.} \ \sum_\alpha f_\alpha (t,x,v) \neq 0\}.$$ Then, there is $C > 0$ such that for any $t \geq 0$ we have $$Q_1(t) \leq C\sqrt{1 + t}.$$
\end{theorem}

The next section is devoted to proving Theorems \ref{Thm1}, \ref{Thm2}, and \ref{Thm3}.  The proofs of Theorems \ref{Thm3} and \ref{Thm5} will be combined since they use similar techniques.  Section $3$ then contains the proof of Theorem \ref{Thm4}, while Section $4$ contains the proof of Theorem \ref{Thm6}.  Throughout the paper, ``$C$'' will denote a generic constant which may change from line to line and depend upon initial data, but not on $t,x$, or $v$.  However, constants with subscripts (e.g., ``$C_0$'') will denote the same numerical value.

\section{Behavior of solutions to (\ref{RVP})}

Throughout this section, we will make great use of the monocharge assumption, which implies $\rho \geq 0$ and thus $\partial_x E \geq 0$.  In addition, we will frequently use the quantity $$M := \iint f_0(x,v) dv dx > 0.$$
In order to prove Theorem \ref{Thm1}, we first need a few lemmas.  The following result shows that characteristics in (\ref{char}) display an order-preserving property due to the monocharge assumption.
\begin{lemma} \label{Lma1}
Let $x \leq x^*$, $v_1 \leq v_1^*$, $\vert v_2 \vert \geq \vert v_2^* \vert$, and $0 \leq t_1 \leq t_2$.  Then, $$X(t_2,t_1,x,v) \leq X(t_2,t_1,x^*,v^*)$$ and $$V_1(t_2,t_1,x,v) \leq V_1(t_2,t_1,x^*,v^*).$$  Also, $x < x^*$ and $t_1 < t_2$ implies $\displaystyle X(t_2,t_1,x,v) < X(t_2,t_1,x^*,v^*)$, and similarly for characteristic momenta $V_1$ if $v_1 < v_1^*$ and $t_1 < t_2$.
\end{lemma}
\begin{proof}
Suppose $x < x^*$, $t_1 < t_2$, and define $$\tau = \sup \{t \in [t_1,t_2] : X(s,t_1,x,v) \leq X(s,t_1,x^*,v^*) \ \forall s \in [t_1,t] \}.$$ Then, for $s \in [t_1,\tau]$, since $\rho \geq 0$, $E_1(t,\cdot)$ is increasing and
\begin{eqnarray*}
\frac{d}{ds} \left [V_1(s,t_1,x^*,v^*) - V_1(s,t_1,x,v) \right ] & = & E_1(s,X(s,t_1,x^*,v^*))\\ & & - \ E_1(s,X(s,t_1,x,v))\\
& \geq & 0.
\end{eqnarray*}
Thus, $$V_1(s,t_1,x^*,v^*) - V_1(s,t_1,x,v) \geq v_1^* - v_1 \geq 0$$ and since $\hat{v}_1=\frac{v_1}{\sqrt{1+\vert v \vert^2}}$ is increasing as a function of $v_1$, we find
\begin{eqnarray*}
\dot{X}(s,t_1,x^*,v^*) & = & \frac{V_1(s,t_1,x^*,v^*)}{\sqrt{1 + V_1(s,t_1,x^*,v^*)^2 + (v_2^*)^2}}\\
& \geq & \frac{V(s,t_1,x,v)}{\sqrt{1 + V_1(s,t_1,x,v)^2 + (v_2^*)^2}}\\
& \geq & \frac{V(s,t_1,x,v)}{\sqrt{1 + V_1(s,t_1,x,v)^2 + v_2^2}}\\
& = & \dot{X}(s,t_1,x,v).
\end{eqnarray*}
Hence, $$X(s,t_1,x^*,v^*) - X(s,t_1,x,v) \geq x^* - x > 0$$ and it follows that $\tau = t_2$.  A similar argument can be used to show the conclusion for characteristic momenta when $v_1 < v_1^*$.  The cases $x^*=x$ or $v^*=v$ follow by continuity with respect to initial conditions.
\end{proof}

We will be concerned with extremal values of position and momentum on the support of $f$ and will make great use of Lemma \ref{Lma1}, so define the quantities
$$W_2 = \sup \{ \vert  v_2 \vert : \exists x, v_1 \in \mathbb{R} \ \mbox{s.t.} \ f_0(x,v) \neq 0 \},$$
$$ s(t) = \overline{\{(x,v_1) \in \mathbb{R}^2: \exists v_2 \in \mathbb{R} \ \mbox{s.t.} \ f(t,x,v) \neq 0 \} },$$
and
$$ \mS(t) =  s(t) \times [-W_2,W_2].$$
Next, we define the largest position and $v_1$ momentum in this set, so let
$$ P_1(t) = \max \{v_1 : \exists x, v_2 \in \mathbb{R} \ \mbox{s.t.} \ (x,v) \in \mS(t) \},$$
$$ R(t) = \max \{x : \exists v \in \mathbb{R}^2 \ \mbox{s.t.} \ (x,v) \in \mS(t) \},$$
and
$$ p_1(t) = \max \{v_1 : \exists v_2 \in \mathbb{R} \ \mbox{s.t.} \ (R(t),v) \in \mS(t) \}.$$
From (\ref{char}) we know that $V_2(s,t,x,v_1,v_2) = v_2$ for any choice of $s,t,x$, and $v_1$.  Thus, the $v_2$-support of $f(t,x,v)$ is constant in time. Hence, let $p_2 = 0$ so that $\vert v_2 \vert \geq \vert p_2 \vert$ for every $v_2$ with $(x,v) \in \mS(t)$.  For the sake of notation, we will write $P_2(t) = P_2 = p_2(t) = p_2 $ for every $t \geq 0$.  Then, from the momentum functions, we define the vectors $P(t) = (P_1(t),P_2(t))$ and $p(t) = (p_1(t),p_2(t))$, and finally the position
$$ r(t) = \max \{x : (x,P(t)) \in \mS(t) \}.$$
Notice that these definitions concern the largest positions and $v_1$ momenta, in addition to the smallest values of $\vert v_2 \vert$ on the support of $f$, and thus, will allow us to utilize Lemma \ref{Lma1}. For intuitive purposes, these quantities are illustrated roughly in Figure $1$ below, which suppresses the dependence on $v_2$.

\begin{figure}[tp]
\begin{center}
  \includegraphics[width=4.5in]{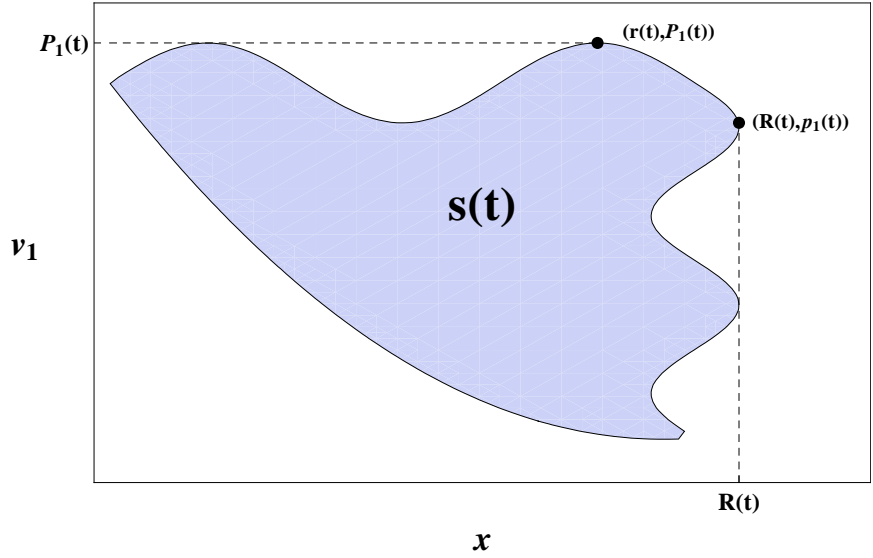}\\
  \caption{Maximal positions and momenta: An example of the support $s(t)$.}\label{AIMS}
  \end{center}
\end{figure}
\noi Since the field $E_1$ only attains values in $[-\frac{1}{2}M,\frac{1}{2}M]$, we have the following lemma.

\begin{lemma} \label{Lma2}
For $t_2 > t_1 \geq 0$,
\begin{equation}
\label{p1}
p_1(t_2) - p_1(t_1) \geq \frac{1}{2}M(t_2 - t_1)
\end{equation}
and $$R(t_2) - R(t_1) = \int_{t_1}^{t_2} \hat{p}_1(s) \ ds.$$
\end{lemma}
\begin{proof}
Let $(X(t),V(t))$ be characteristics of $f$ as in (\ref{char}).  Partition $[t_1,t_2]$ into $n$ subintervals $t_1 = s_0 < s_1 < \cdots < s_n = t_2$ and let $$s_{k+1} - s_k = \Delta s = \frac{t_2 - t_1}{n}$$ for all $k = 0,..., n-1.$  Then, for any $k$
\begin{equation}
\label{one} X(s_{k+1}, s_{k+1}, R(s_{k+1}),p(s_{k+1})) = R(s_{k+1}) \geq X(s_{k+1}, s_k, R(s_k),p(s_k))
\end{equation}
and hence the function $\xi(t) := X(t, s_{k+1}, R(s_{k+1}),p(s_{k+1})) - X(t,s_k, R(s_k),p(s_k))$ satisfies $\xi(s_{k+1}) \geq 0$.  Similarly,
$$X(s_k, s_{k+1}, R(s_{k+1}),p(s_{k+1})) \leq R(s_k) = X(s_k,s_k, R(s_k),p(s_k))$$ and thus $\xi(s_k) \leq 0$.  So, by the Intermediate Value Theorem there exists $\tau \in [s_k, s_{k+1}]$ such that $\xi(\tau) = 0$ and $$ X(\tau, s_{k+1}, R(s_{k+1}),p(s_{k+1})) =  X(\tau, s_k, R(s_k),p(s_k)).$$  If we suppose $$V_1(\tau, s_k, R(s_k), p(s_k)) >  V_1(\tau, s_{k+1}, R(s_{k+1}),p(s_{k+1}))$$ then by Lemma \ref{Lma1} and since $$\vert V_2(\tau, s_k, R(s_k), p(s_k)) \vert = \vert p_2(s_k) \vert = \vert p_2(s_{k+1}) \vert = \vert V_2(\tau,s_{k+1}, R(s_{k+1}),p(s_{k+1})) \vert $$ we would have $$X(s, s_k, R(s_k), p(s_k)) >  X(s, s_{k+1}, R(s_{k+1}),p(s_{k+1}))$$ for all $s > \tau$, which contradicts (\ref{one}).  Hence, we find $$V_1(\tau, s_k, R(s_k), p(s_k)) \leq  V_1(\tau, s_{k+1}, R(s_{k+1}),p(s_{k+1})).$$ By Taylor's theorem,
\begin{eqnarray*}
V_1(\tau, s_k, R(s_k), p(s_k)) & \geq & p_1(s_k) + E_1(s_k, R(s_k))(\tau - s_k) - C (\Delta s)^2 \\
& = & p_1(s_k) + \frac{1}{2}M(\tau - s_k) - C (\Delta s)^2
\end{eqnarray*}
and
\begin{eqnarray*}
V_1(\tau, s_{k+1}, R(s_{k+1}),p(s_{k+1})) & \leq & p_1(s_{k+1}) + E_1(s_{k+1}, R(s_{k+1}))(\tau - s_{k+1})\\
& & + \ C (\Delta s)^2 \\
& = & p_1(s_{k+1}) + \frac{1}{2}M(\tau - s_{k+1}) + C (\Delta s)^2.
\end{eqnarray*}
Hence,
\begin{eqnarray*}
0 & \leq & V_1(\tau, s_{k+1}, R(s_{k+1}),p(s_{k+1})) - V_1(\tau, s_k, R(s_k), p(s_k))\\
& \leq & p_1(s_{k+1}) + \frac{1}{2}M(\tau - s_{k+1}) + C (\Delta s)^2\\
& & - \ \left ( p_1(s_k) + \frac{1}{2}M(\tau - s_k) - C (\Delta s)^2 \right ) \\
& = & p_1(s_{k+1}) - p_1(s_k) - \frac{1}{2}M(s_{k+1} - s_k) + C (\Delta s)^2
\end{eqnarray*}
and we find
\begin{eqnarray*}
p_1(t_2) - p_1(t_1) & = & \sum_{k=0}^{n-1} \left [ p_1(s_{k+1}) - p_1(s_k) \right ] \\
& \geq & \sum_{k=0}^{n-1} \left [ \frac{1}{2}M(s_{k+1} - s_k) - C (\Delta s)^2 \right ] \\
& = & \frac{1}{2}M(t_2 - t_1) - C n (\Delta s)^2.
\end{eqnarray*}
This holds for any $n \in \mathbb{N}$, and thus it follows that $$p_1(t_2) - p_1(t_1) \geq \frac{1}{2}M(t_2 - t_1).$$ In addition, $p_1$ and $\hat{p}_1$ are both increasing and hence integrable.  Next,
\begin{eqnarray*}
\lefteqn{\frac{R(s_{k+1}) - R(s_k)}{s_{k+1} - s_k} - \hat{p}_1(s_k)} \\ & & \geq \frac{X(s_{k+1}, s_k, R(s_k), p(s_k)) - R(s_k)}{s_{k+1} - s_k} - \hat{p}_1(s_k) \\
& & = \frac{X(s_{k+1}, s_k, R(s_k), p(s_k)) - X(s_k, s_k, R(s_k), p(s_k))}{s_{k+1} - s_k}\\
& & \qquad  - \ \dot{X}(s_k, s_k, R(s_k), p(s_k)) \\
& & \geq - C \Delta s
\end{eqnarray*}
and similarly
\begin{eqnarray*}
\lefteqn{\frac{R(s_{k+1}) - R(s_k)}{s_{k+1} - s_k} - \hat{p}_1(s_{k+1})}\\ & & \leq  \frac{R(s_{k+1}) - X(s_k, s_{k+1}, R(s_{k+1}), p(s_{k+1}))}{s_{k+1} - s_k} - \hat{p}_1(s_{k+1}) \\
& & = \frac{X(s_{k+1}, s_{k+1}, R(s_{k+1}), p(s_{k+1})) - X(s_k, s_{k+1}, R(s_{k+1}), p(s_{k+1}))}{s_{k+1} - s_k}\\
& & \qquad - \ \dot{X}(s_{k+1}, s_{k+1}, R(s_{k+1}), p(s_{k+1})) \\
& & \leq C \Delta s.
\end{eqnarray*}
Multiplying these inequalities by $s_{k+1} - s_k$ and summing over $k = 0, ..., n-1$, we find
$$ R(t_2) - R(t_1) - \sum_{k=0}^{n-1} \hat{p}_1(s_k)(s_{k+1} - s_k) \geq - C(t_2 - t_1) \Delta s$$
and
$$ R(t_2) - R(t_1) - \sum_{k=0}^{n-1} \hat{p}_1(s_{k+1})(s_{k+1} - s_k) \leq C(t_2 - t_1) \Delta s.$$
Noting the missing $t_2$ term in the sum of the first inequality and the missing $t_1$ term in the second, we find  $$\left \vert R(t_2) - R(t_1) - \sum_{k=0}^{n-1} \hat{p}_1(s_{k+1})(s_{k+1} - s_k) \right \vert \leq C \Delta s + \vert \hat{p}_1(t_1) \vert \Delta s + \vert \hat{p}_1(t_2) \vert \Delta s.$$  Finally, this holds for any $n \in \mathbb{N}$, and it follows that $$R(t_2) - R(t_1) = \int_{t_1}^{t_2} \hat{p}_1(s) \ ds.$$
\end{proof}

\begin{remark}
The result (\ref{p1}) yields
\begin{equation}
\label{P1}
P_1(t) \geq p_1(t) \geq p_1(0) + \frac{1}{2}Mt.
\end{equation}
\end{remark}

\noi We may proceed in a similar manner to determine the behavior of $r$ and $P_1$.

\begin{lemma} \label{Lma3}
For $t_2 \geq t_1 \geq 0$, we have
\begin{equation}
\label{rstar} r(t_2) - r(t_1) \geq \int_{t_1}^{t_2} \widehat{P}_1(s) \ ds,
\end{equation}
\begin{equation}
\label{Einc}
E_1(t_2, r(t_2)) \geq E_1(t_1, r(t_1)),
\end{equation}
and
\begin{equation}
\label{Pstar} P_1(t_2) - P_1(t_1) = \int_{t_1}^{t_2} E_1(s, r(s)) \ ds.
\end{equation}
\end{lemma}
\begin{proof}
Proceeding as in the previous lemma, $$P_1(s_{k+1}) \geq V_1(s_{k+1},s_k,r(s_k),P(s_k))$$ and $$P_1(s_k) \geq V_1(s_k,s_{k+1},r(s_{k+1}),P(s_{k+1})).$$  Thus, there is $\tau \in [s_k, s_{k+1}]$ such that $$V_1(\tau,s_{k+1},r(s_{k+1}),P(s_{k+1})) = V_1(\tau, s_k,r(s_k),P(s_k)).$$ Using Lemma \ref{Lma1} again and noticing $$\vert V_2(\tau,s_k,r(s_k),P(s_k)) \vert = \vert P_2(s_k) \vert = \vert P_2(s_{k+1}) \vert = \vert V_2(\tau,s_{k+1},r(s_{k+1}),P(s_{k+1})) \vert,$$ we find $$X(\tau,s_{k+1},r(s_{k+1}),P(s_{k+1})) \geq X(\tau, s_k,r(s_k),P(s_k))$$ follows immediately.  Now, $$ X(\tau,s_{k+1},r(s_{k+1}),P(s_{k+1})) \leq r(s_{k+1}) + \widehat{P}_1(s_{k+1})(\tau - s_{k+1}) + C (\Delta s)^2$$ and $$ X(\tau,s_k,r(s_k),P(s_k)) \geq r(s_k) + \widehat{P}_1(s_k)(\tau - s_k) - C (\Delta s)^2,$$ so $$0 \leq r(s_{k+1}) - r(s_k) - \widehat{P}_1(s_{k+1})(s_{k+1} - \tau) - \widehat{P}_1(s_k)(\tau - s_k) + C (\Delta s)^2$$ and hence $$r(t_2) - r(t_1) \geq \sum_{k=0}^{n-1} \left [ \widehat{P}_1(s_{k+1})(s_{k+1} - \tau) + \widehat{P}_1(s_k)(\tau - s_k)\right ] - Cn (\Delta s)^2.$$  Using the Lipschitz continuity of $\widehat{P}_1$, we find $$r(t_2) - r(t_1) \geq \sum_{k=0}^{n-1} \widehat{P}_1(s_k)(s_{k+1} - s_k) - Cn (\Delta s)^2.$$ and (\ref{rstar}) follows.  Next, let $$\bar{r}(s) = r(t_1) + \int_{t_1}^s \widehat{P}_1(\tau) d\tau.$$  From the definition of $P(t)$ and as in the proof of Lemma \ref{Lma1}, we have for $(x,v) \in \mS(t)$, $$ \hat{v}_1 = \frac{v_1}{\sqrt{1 + v_1^2 + v_2^2}} \leq \frac{P_1(t)}{\sqrt{1 + P_1^2(t) + P_2^2}} = \widehat{P}_1(t). $$Since $r(t_2) \geq \bar{r}(t_2)$, and $E_1(t,\cdot)$ is increasing, we find
\begin{eqnarray*}
E_1(t_2,r(t_2)) - E_1(t_1,r(t_1)) & \geq & E_1(t_2,\bar{r}(t_2)) - E_1(t_1,\bar{r}(t_1)) \\
& = & \left. \int_{t_1}^{t_2} ( \partial_t E_1 + \bar{r}^\prime(t) \partial_x E_1) \right \vert_{(t,\bar{r}(t))} \ dt \\
& = & \int_{t_1}^{t_2} \int f(t,\bar{r}(t), v) \left ( \widehat{P}_1(t) - \hat{v}_1 \right ) \ dv dt\\
& \geq & 0
\end{eqnarray*}
and (\ref{Einc}) holds.  Next,
\begin{eqnarray*}
\lefteqn{\frac{P_1(s_{k+1}) - P_1(s_k)}{s_{k+1} - s_k} - E_1(s_k, r(s_k))}\\ & & \geq \frac{V_1(s_{k+1}, s_k, r(s_k),P(s_k)) - V_1(s_k, s_k, r(s_k),P(s_k))}{s_{k+1} - s_k}\\
& & \qquad - \ \dot{V}_1(s_k, s_k, r(s_k),P(s_k)) \\
& & \geq - C\Delta s
\end{eqnarray*}
and similarly $$\frac{P_1(s_{k+1}) - P_1(s_k)}{s_{k+1} - s_k} - E_1(s_{k+1}, r(s_{k+1})) \leq C \Delta s.$$ Hence, $$P_1(t_2) - P_1(t_1) - \sum_{k=0}^{n-1} E_1(s_k, r(s_k))(s_{k+1} - s_k) \geq -Cn (\Delta s)^2$$ and $$P_1(t_2) - P_1(t_1) - \sum_{k=0}^{n-1} E_1(s_{k+1}, r(s_{k+1}))(s_{k+1} - s_k) \leq Cn (\Delta s)^2.$$  Using (\ref{Einc}) and the boundedness of $E_1$ yields $$\left \vert P_1(t_2) - P_1(t_1) - \sum_{k=0}^{n-1} E_1(s_{k+1}, r(s_{k+1}))(s_{k+1} - s_k) \right \vert \leq Cn (\Delta s)^2 + C \Delta s$$ for every $n \in \mathbb{N}$ and (\ref{Pstar}) follows.
\end{proof}

\begin{corollary}
$\displaystyle E_1(s,r(s)) \rightarrow \frac{1}{2}M$ as $s \rightarrow \infty$.
\end{corollary}
\begin{proof}
From the previous results (\ref{P1}) and (\ref{Pstar}), we may deduce $$P_1(0) + \int_0^t E_1(s,r(s)) \ ds = P_1(t) \geq p_1(t) \geq p_1(0) + \frac{1}{2}Mt$$ and hence for all $t\geq 0$, $$\int_0^t E_1(s,r(s)) \ ds \geq \frac{1}{2}Mt - C.$$  Additionally, we know that $E_1(s,r(s))$ is increasing and $E_1(s,r(s)) \leq \frac{1}{2}M$ for all $s \geq 0$. Hence, $E_1(s,r(s))$ must converge as $s \rightarrow \infty$ and its limit can be no greater and no less than $\frac{1}{2}M$.
\end{proof}

Finally, we may prove the non-decay result.  In the proof of Theorem \ref{Thm1} the generic constant ``$C$'' may depend upon the fixed variable $t>0$.

\begin{proof}[Proof of Theorem \ref{Thm1}]
Note that we may rewrite $E_1$ as $$E_1 = \frac{1}{2} \int_{-\infty}^x \rho dy - \frac{1}{2}\int_x^\infty \rho dy = \int_{-\infty}^x \rho dy - \frac{1}{2}M$$ so that $$\int_{-\infty}^x \int f(t,y,v) \ dv dy = \int_{-\infty}^x \rho(t,y) dy = \frac{1}{2}M + E_1(t,x).$$  By the corollary, we may choose $t$ large enough so that $$E_1(t,r(t)) > 0.4M$$ and hence $$\int_{-\infty}^{r(t)} \iint_{-\infty}^{P_1(t)} f(t,x,v) dv_1 dv_2 dx > 0.9M.$$  By the Intermediate Value Theorem, there is $\epsilon > 0$ such that $$\int_{-\infty}^{r(t)-\epsilon} \iint_{-\infty}^{P_1(t)-\epsilon} f(t,x,v) dv_1 dv_2 dx = 0.8M.$$
Additionally, notice that $$\int_{r(t)-\epsilon}^{r(t)} \iint_{P_1(t)-\epsilon}^{P_1(t)} f(t,x,v) dv_1 dv_2 dx > 0$$ since $(r(t),P(t)) \in \mS(t)$.  If $(X(t),V(t))$ is any characteristic then
\begin{eqnarray*}
\lefteqn{\frac{d}{ds} \int_{-\infty}^{X(s)} \iint_{-\infty}^{V_1(s)} f(s,x,v) dv_1 dv_2 dx}\\  & & = \int_{-\infty}^{X(s)} \int f(s,x,V_1(s),v_2) \left ( E_1(s,X(s)) - E_1(s,x) \right )  d v_2dx\\
& & \qquad + \ \iint_{-\infty}^{V_1(s)} f(s,X(s),v) \left ( \widehat{V}_1(s) - \hat{v}_1 \right ) dv_1 dv_2\\
& & \geq 0
\end{eqnarray*}
Similarly, we find $$\frac{d}{ds} \int_{X(s)}^{\infty} \iint_{V_1(s)}^{\infty} f(s,x,v) dv_1dv_2 dx \geq 0.$$  We take $$X(s) = X(s,t,r(t)-\epsilon, P_1(t)-\epsilon,P_2),$$ $$V_1(s) = V_1(s,t,r(t)-\epsilon, P_1(t)-\epsilon,P_2),$$ and $$V_2(s) = P_2(s)=P_2.$$
Then, from above
\begin{equation}
\label{fintleft}
\int_{-\infty}^{X(s)} \iint_{-\infty}^{V_1(s)} f(s,x,v) dv_1 dv_2 dx \geq 0.8M
\end{equation} and
\begin{equation}
\label{fintright}
\int_{X(s)}^{\infty} \iint_{V_1(s)}^{\infty} f(s,x,v) dv_1 dv_2 dx \geq  C> 0
\end{equation} for $s \geq t$.  From (\ref{fintleft}) and the definition of $E_1$ it follows that $$E_1(s,X(s)) \geq 0.3M$$ and $$V_1(s) \geq 0.3M(s-t).$$ Since $h(v_1,v_2) = \hat{v}_1$ satisfies $\displaystyle \frac{\partial h}{\partial v_1} = (1 + \vert v \vert^2)^{-3/2} (1+v_2^2) > 0$ and is thus increasing in $v_1$, we find for $s \geq t$
\begin{equation}
\label{vhatlb}
\widehat{V}_1(s) \geq \frac{0.3M(s-t)}{\sqrt{1 + [0.3M(s-t)]^2 + P_2^2}}.
\end{equation}  Put $A = 0.3M(s-t)$.  From (\ref{vhatlb}) and the compact $v_2$-support of $f$ we find
\begin{eqnarray*}
1 - \widehat{V}_1(s) & \leq & \frac{\sqrt{1 + A^2 + P_2^2} - A}{\sqrt{1 + A^2 + P_2^2}} \\
& = & \frac{1 + P_2^2}{\sqrt{1 + A^2 + P_2^2} \left ( \sqrt{1 + A^2 + P_2^2} + A \right )}\\
& \leq & \frac{C}{1 + A^2}
\end{eqnarray*}
and thus
\begin{equation}
\label{V1bd}
1 - \widehat{V}_1(s) \leq \frac{C}{1 + [0.3M(s-t)]^2}
\end{equation}
for $s \geq t$. Because $\vert \widehat{V}_1(s) \vert < 1$, we see from the compact $x$-support of $f$ that $f(t,x,v) = 0$ for $x \geq R(0) + t$.  Thus, $\rho(t,x) = 0$ for $x \geq R(0) + t$.  Finally, (\ref{fintright}) yields
\begin{eqnarray*}
C & \leq & \int_{X(s)}^\infty \rho(s,x) \ dx \\
& = & \int_{X(s)}^{R(0) + s} \rho(s,x) \ dx \\
& \leq & \Vert \rho(s) \Vert_p \left ( R(0) + s - X(s) \right )^{1-\frac{1}{p}}.
\end{eqnarray*}
However, using (\ref{V1bd})
\begin{eqnarray*}
R(0) + s - X(s) & = & R(0) + t - X(t) + (s - t) - \left( X(s) - X(t) \right ) \\
& = & R(0) + t - X(t) + \int_t^s \left ( 1 - \widehat{V}_1(\tau) \right ) d\tau \\
& \leq & C +  \int_t^\infty \frac{C}{1 + [0.3M(\tau-t)]^2} d\tau \\
& \leq & C
\end{eqnarray*}
where $C$ may depend upon the fixed time $t$ but not on $s \geq t$.  So, we find $$C \leq \Vert \rho(s) \Vert_p$$ and the proof of the theorem is complete.
\end{proof}

\begin{proof}[Proof of Theorem \ref{Thm2}]
We will sometimes use an abbreviated notation for characteristics, namely $$X(s) = X(s,t,x,v)$$ and similarly for $V(s)$.  Let $t > 0$ and define $$\tau^* = \inf \{ \tau \in [0,t) : \frac{\partial X}{\partial v_1}(s,t,x,v) < 0 \quad \forall s \in (\tau,t) \}.$$  Note that (\ref{char}) implies $$\frac{d}{ds} \left. \left ( \frac{\partial X}{\partial v_1} \right ) \right \vert_{s=t} = \left (1 + \vert v \vert ^2 \right)^{-3/2} (1 + v_2^2) > 0$$ and $$ \left. \frac{\partial X}{\partial v_1}  \right \vert_{s=t} = 0$$ so that $\tau^*$ is well-defined.  Next, for every $s \in [\tau^*,t]$ we know $\displaystyle \frac{\partial X}{\partial v_1}(s) \leq 0$ so $$ \frac{d}{ds} \left ( \frac{\partial V_1}{\partial v_1} \right ) = \partial_x E_1 \ \frac{\partial X}{\partial v_1} \leq 0,$$ and thus $$\frac{\partial V_1}{\partial v_1}(s) \geq \left. \frac{\partial V_1}{\partial v_1} \right \vert_{s=t} = 1.$$ Using this lower bound, $$ \frac{d}{ds} \left ( \frac{\partial X}{\partial v_1} \right ) =  \frac{1 + v_2^2}{(1 + \vert V \vert^2)^{3/2} } \frac{\partial V_1}{\partial v_1} \geq \frac{1 + v_2^2}{(1 + \vert V \vert^2)^{3/2}} > 0,$$ and hence $$\frac{\partial X}{\partial v_1}(\tau^*,t,x,v) <0  .$$  From this, it follows that $\tau^* = 0$ and hence $\displaystyle 1 \leq  \frac{\partial V_1}{\partial v_1}(s)$ for $0 \leq s \leq t$.  Now a change of variable in $\rho$ will complete the theorem:

\begin{eqnarray*}
\rho(t,x) & = & \int f(0,X(0),V(0)) \ dv \\
& = & \int \frac{f_0(X(0),V(0))}{\frac{\partial V_1}{\partial v_1}(0)} \frac{\partial V_1}{\partial v_1}(0) \ dv\\
& \leq & \iint F_0(V_1(0,t,x,v), v_2) \frac{\partial V_1}{\partial v_1}(0) \ dv_1 dv_2\\
& = & \int F_0(w) \ dw.
\end{eqnarray*}
\end{proof}

\begin{proof}[Proof of Theorems \ref{Thm3} and \ref{Thm5}]
We will first prove the result for (\ref{VP}) and then comment on the necessary changes to adapt the proof to (\ref{RVP}).  Since we have defined the characteristics for (\ref{RVP}) in (\ref{char}), let us do the same for (\ref{VP}), noting that the only difference occurs in the equation for $\displaystyle \frac{dX}{ds}$:
\begin{equation}
\label{charVP} \left. \begin{array}{ccc}
& & \displaystyle \frac{dX}{ds} = V_1 (s) \\
& & \displaystyle \frac{dV_1}{ds} = E_1(s,X(s)) \\
& & V_2(s,t,x,v) = v_2 \\
\\
& & X(t,t,x,v) = x \\
& & V_1(t,t,x,v) = v_1
\end{array} \right \}
\end{equation}
Now, by the definition of $E_1$ in (\ref{VP}) and the well-known conservation of charge identity, we find $\vert E_1(t,x) \vert \leq \frac{1}{2}M$ for every $x\in\mathbb{R}, t > 0$ and $\displaystyle \lim_{x \rightarrow \pm \infty} E_1(t,x) = \pm \frac{1}{2}M$.  Using (\ref{charVP}) and the compact support of $f_0$ in $v_1$, $$ \vert V_1(t) \vert = \left \vert V_1(0) + \int_0^t E_1(s,X(s)) \ dx \right \vert \leq C + \frac{1}{2}Mt$$ for any characteristics $(X(t), V(t))$ along which $f(t,X(t),V(t)) \neq 0$.  Taking the supremum over all such characteristics, we find $$Q_1(t) \leq Ct$$ for $t$ large enough.  The result can be seen to hold for (\ref{RVP}), as well, since the equation for $\displaystyle \frac{dV_1}{ds}$ in (\ref{char}) does not differ from that of (\ref{charVP}) and the same bound $\vert E_1(t,x) \vert \leq \frac{1}{2}M$ again follows from charge conservation.\\

To obtain the contrasting inequality for (\ref{VP}), multiply the Vlasov equation by $v_1 E_1$ and integrate over $v$ to find
$$ \int v_1 E_1 \partial_t f dv + \int v_1^2 E_1 \partial_x f dv + \int v_1 E_1^2 \partial_{v_1} f \ dv = 0,$$  which is equivalent to
\begin{equation}
\label{VlasovE}
\partial_t \int v_1 E_1f \ dv - (\partial_t E_1) \int v_1 f \ dv + \partial_x \int v_1^2 E_1 f dv - \rho \int v_1^2 f \ dv = \int E_1^2 f dv.
\end{equation}
Let $j_1(t,x) = \int v_1 f(t,x,v) dv.$ As usual, integrating the Vlasov equation of (\ref{VP}) in $v$ yields the continuity equation $$\partial_t \rho + \partial_x j_1 = 0.$$  Using this, we see that $\partial_t E_1 = -j_1$ and integrating (\ref{VlasovE}) in $x$ yields
$$ \frac{d}{dt} \iint v_1 E_1f \ dv dx = \iint \rho v_1^2 f \ dv dx - \int j_1^2 \ dx + \iint E_1^2 f \ dv dx$$ or
\begin{equation}
\label{Thm2star}
\frac{d}{dt} \iint v_1 E_1f \ dv dx = \int \left ( \rho \int v_1^2 f \ dv - j_1^2 \right ) dx + \int \rho E_1^2 \ dx.
\end{equation}
Now, by Cauchy-Schwarz the first term on the right side of (\ref{Thm2star}) is nonnegative since
$$\left \vert \int v_1 f \ dv \right \vert \leq \left ( \int v_1^2 f \ dv \right )^{1/2} \left (\int f \ dv \right )^{1/2}$$ and thus $$\rho \int v_1^2 f dv =  \left (\int f \ dv \right ) \left (\int v_1^2 f \ dv \right ) \geq \left (\int v_1 f \ dv \right )^2 = j_1^2.$$
The second term on the right side of (\ref{Thm2star}) can be simplified as well
$$ \int \rho E_1^2 \ dx = \frac{1}{3} \int \partial_x (E_1^3) \ dx = \frac{1}{3} \left ( \frac{M^3}{8} - \frac{(-M)^3}{8} \right ) = \frac{M^3}{12}.$$
 Thus, from (\ref{Thm2star}) we deduce $$\frac{d}{dt} \iint v_1 f E_1 \ dv dx \geq \frac{M^3}{12}$$  and by integrating in time and using charge conservation, i.e. $\Vert f(t) \Vert_1 = \Vert f_0 \Vert_1  = M$,
\begin{eqnarray*}
\frac{M^3}{12}t & \leq & \iint E_1(t,x) v_1 f(t,x,v) \ dv dx - \iint E_1(0,x) v_1 f_0(x,v) \ dv dx \\
& \leq & C\left (1 + \frac{1}{2}M^2 Q_1(t) \right ) \\
& \leq & C(1 + Q_1(t)).
\end{eqnarray*}
Hence, for $t$ large enough $Q_1(t) \geq Ct$.  For (\ref{RVP}) the analogous claim follows immediately from (\ref{P1}) since $$Q_1(t) \geq P_1(t) \geq p_1(t) \geq p_1(0) + \frac{1}{2}Mt.$$ Hence, for $t$ large this implies $$Q_1(t) \geq Ct$$ and the proof is complete.
\end{proof}

\section{Behavior of solutions to (\ref{RVPN})}

\begin{proof}[Proof of Theorem \ref{Thm4}]
We first obtain \emph{a priori} bounds along the light cone as in \cite{GlaSch}.  Put $$e(t,x) = \int \sqrt{1 + \vert v \vert^2} \sum_\alpha f_\alpha(t,x,v) \ dv + \frac{1}{2} E_1^2(t,x),$$ $$m(t,x) = \int v_1 \sum_\alpha f_\alpha(t,x,v) \ dv.$$ As is well known, the local conservation of energy can be obtained by summing all of the Vlasov equations and integrating in $v$ to find   $$ \partial_t e + \partial_x m = 0.$$ An integration over all $x$ then yields a bound on the total energy $$\int e(t,x) \ dx = \int e(0,x) \ dx \leq C.$$  Here, we must point out that this identity, and those which follow, hold only under the assumption of neutrality (\ref{Neutrality}).  If the plasma is not neutral (as is the case for the (\ref{RVP}) system), then $E_1(t,x) \rightarrow \pm \frac{1}{2}M$ as $x \rightarrow \pm \infty$ for every $t \geq 0$, where $M \neq 0$ is the total charge.  In these cases, $E(t,\cdot) \not\in L^2(\mathbb{R})$ and the total potential energy is infinite.  Now, for $T>0$ we integrate the local energy identity over the forward cone originating from the point $(0,x)$ $$0 = \int_0^T \int_{x-t}^{x+t}  \left ( \partial_t e + \partial_y m \right ) \ dy dt$$ and use Green's Theorem to obtain
\begin{equation}
\label{coneT}
\int_0^T (e - m)(t,x+t) \ dt + \int_0^T (e + m)(t,x-t) \ dt = \int_{x-T}^{x+T} e(T,y) \ dy.
\end{equation}
Since $e \pm m = \int (\sqrt{1 + \vert v \vert^2} \pm v_1) \sum_\alpha f_\alpha \ dv + \frac{1}{2} E_1^2 \geq 0$ and $$\int_{x-T}^{x+T} e(T,y) \ dy \leq \int e(T,y) \ dy = \int e(0,y) \ dy,$$ we have a bound on each term on the left side of (\ref{coneT}) for all $T \geq 0$.  Hence, the integral of $e \pm m$ is bounded along the rays of the cone:
\begin{equation}
\label{conefg}
\int_0^\infty \left [ \int \left (\sqrt{1 + \vert v \vert^2} - v_1 \right ) \sum_\alpha f_\alpha dv + \frac{1}{2} E_1^2 \right ] (t,x+t) \ dt \leq C
\end{equation}
and
$$\int_0^\infty \left [ \int \left (\sqrt{1 + \vert v \vert^2} + v_1 \right ) \sum_\alpha f_\alpha dv + \frac{1}{2} E_1^2 \right ] (t,x-t) \ dt \leq C$$
for any $x \in \mathbb{R}$.  Next, we use a lemma which allows us to utilize the cone estimate in a way that controls the $v$-integral of $\sum_\alpha f_\alpha$.

\begin{lemma} \label{Lma4}
Let $$\sigma_-(t,x) = \int \left (\sqrt{1 + \vert v \vert^2} - v_1\right ) \sum_\alpha f_\alpha(t,x,v) \ dv,$$ $$\sigma_+(t,x) = \int \left (\sqrt{1 + \vert v \vert^2} + v_1\right ) \sum_\alpha f_\alpha(t,x,v) \ dv,$$ and $$k(t,x) = \int \sqrt{1 + \vert v \vert^2} \sum_\alpha f_\alpha(t,x,v) \ dv.$$ Then, we have $$\int \sum_\alpha f_\alpha dv \leq 3 \sqrt{\sigma_- k}$$ and $$\int \sum_\alpha f_\alpha dv \leq 3 \sqrt{\sigma_+ k}.$$
\end{lemma}
\begin{proof}
We will show the former inequality, as the proof of the latter is similar.  Let $R \geq 0$ and split the $v$-integral:
\begin{eqnarray*}
\int \sum_\alpha f_\alpha \ dv & = & \int_{\vert v \vert < R} \sum_\alpha f_\alpha \ dv + \int_{\vert v \vert > R} \sum_\alpha f_\alpha \ dv \\
& \leq & \int_{\vert v \vert < R} \sum_\alpha f_\alpha \ dv + \frac{k}{\sqrt{1 + R^2}}.
\end{eqnarray*}
Consider $\vert v \vert \leq R$ and notice
\begin{eqnarray*}
\sqrt{1 + \vert v \vert^2} - v_1 & = & \frac{1 + v_2^2}{\sqrt{1 + \vert v \vert^2} + v_1} \\
& \geq & \frac{1 + v_2^2}{2 \sqrt{1 + \vert v \vert^2} } \\
& \geq & \frac{1}{2 \sqrt{1 + \vert v \vert^2} } \\
& \geq & \frac{1}{2 \sqrt{1 + R^2} } \\
\end{eqnarray*}
So, we find
\begin{eqnarray*}
\int \sum_\alpha f_\alpha \ dv & \leq & \int_{\vert v \vert \leq R} \sum_\alpha f_\alpha \cdot \left ( \sqrt{1 + \vert v \vert^2} - v_1 \right ) \left ( 2\sqrt{1 + R^2} \right ) \ dv + \frac{k}{\sqrt{1 + R^2}}\\
& \leq & 2 \sqrt{1 + R^2} \cdot \sigma_-(t,x) + \frac{k(t,x)}{\sqrt{1 + R^2}}.
\end{eqnarray*}
If $0 < \sigma_- \leq k$ take $\displaystyle R = \sqrt{\frac{k}{\sigma_-} - 1}$ so that $\displaystyle \sqrt{1 + R^2} = \sqrt{\frac{k}{\sigma_-}}$ and
$$ \int \sum_\alpha f_\alpha \ dv \leq 2 \sqrt{\frac{k}{\sigma_-}} \cdot \sigma_- + \frac{k}{\sqrt{k / \sigma_-}} = 3 \sqrt{\sigma_- k}.$$  Now, if $k < \sigma_-$ then $$\int \sum_\alpha f_\alpha \ dv < \int \sqrt{1 + \vert v \vert^2} \sum_\alpha f_\alpha \ dv = k < \sqrt{\sigma_- k} \leq 3 \sqrt{\sigma_- k}.$$ Finally, if $\sigma_- = 0$ then since $f_\alpha \geq 0$ for all $\alpha = 1,...,N$ $$\int \sum_\alpha f_\alpha \ dv = 0 = 3 \sqrt{\sigma_- k}$$ and the proof is complete.
\end{proof}

Since we are now considering the relativistic, neutral system (\ref{RVPN}), recall that the characteristic equations are similar to that of (\ref{char}) but must depend upon the charge of the corresponding species:
\begin{equation}
\label{charRVPN} \left. \begin{array}{ccc}
& & \displaystyle \frac{dX_\alpha}{ds} = \widehat{V}_{1\alpha} (s) \\
& & \displaystyle \frac{dV_{1\alpha}}{ds} = e_\alpha E_1(s,X(s)) \\
& & V_{2\alpha} (s,t,x,v) = v_2 \\
\\
& & X_\alpha(t,t,x,v) = x \\
& & V_{1\alpha}(t,t,x,v) = v_1
\end{array} \right \}
\end{equation}

Now, let $\alpha$ be given and choose characteristics $(X(t),V(t))$ as in (\ref{charRVPN}) such that $f_\alpha(t,X(t),V(t)) \neq 0$.  For brevity, we suppress the characteristic dependence upon the choice of $\alpha$.  Let $$ C_0  = \sup \{ \vert v_1 \vert : \exists t \in [0,1], x, v_2 \in \mathbb{R} \ \mbox{with} \ f_\alpha (t,x,v) \neq 0 \}.$$
and suppose that $t > 0$ and $V_1(t) > 2C_0$.  Define $$\Delta = \sup \left \{ \tau \in (0,t]: V_1(s) \geq \frac{1}{2}V_1(t) \ \forall s \in [t -\tau,t] \right \}$$
and notice that $$V_1(t-\Delta) \geq \frac{1}{2} V_1(t) > C_0$$ so that $t -\Delta > 1$ and
\begin{equation}
\label{V1}
V_1(t-\Delta) = \frac{1}{2}V_1(t).
\end{equation}
Define $$X_c(s) = X(t) + s- t.$$ Then, for $t-\Delta \leq s \leq t$,
\begin{eqnarray*}
\left \vert \frac{d}{ds} ( X(s) - X_c(s) ) \right \vert & = & 1 - \frac{V_1(s)}{\sqrt{1 + \vert V(s) \vert^2}}\\
& = & \frac{1 + V_2(s)^2}{\sqrt{1 + \vert V(s) \vert^2} (\sqrt{1 + \vert V(s) \vert^2} + V_1(s))} \\
& \leq & 4 \frac{1 + V_2(0)^2}{V_1(t)^2} \\
& \leq & \frac{C}{V_1(t)^2}.
\end{eqnarray*}
Hence, integrating in time over $[t-\Delta,t]$ we find
\begin{equation}
\label{Xcone}
\vert X(s) - X_c(s) \vert \leq \frac{C\Delta}{V_1(t)^2}.
\end{equation}
Integrating (\ref{charRVPN}) and using the cone estimate (\ref{conefg}), we find
\begin{eqnarray*}
V_1(t) & = & V_1(t-\Delta) + e_\alpha \int_{t-\Delta}^t E_1(s,X_c(s)) \ ds \\
& & \ + \ e_\alpha \int_{t-\Delta}^t ( E_1(s,X(s)) - E_1(s,X_c(s))) \ ds \\
& \leq & V_1(t-\Delta) +  \left (\max_\alpha \vert e_\alpha \vert \right) \left [ \Delta^{1/2} \left ( \int_{t-\Delta}^t E_1^2(s,X_c(s)) \ ds \right )^{1/2} \right. \\
& \ & \ \left. + \ \int_{t-\Delta}^t \int_{X_c(s)}^{X(s)} \int \sum_\alpha \vert e_\alpha \vert f_\alpha (s,x,v) \ dv dx ds \right ] \\
& \leq & V_1(t-\Delta) + C\Delta^{1/2} + C \int_{t-\Delta}^t \int_{X_c(s)}^{X(s)} \int \sum_\alpha f_\alpha (s,x,v) \ dv dx ds.
\end{eqnarray*}
We use Lemma \ref{Lma4} and (\ref{Xcone}) on the last term to find
\begin{equation}
\label{sqrtks}
\int_{t-\Delta}^t \int_{X_c(s)}^{X(s)} \int \sum_\alpha f_\alpha \ dv dx ds \leq 3 \int_{t-\Delta}^t \int_{X_c(s)}^{X_c(s) + C \Delta V_1(t)^{-2}} \sqrt{\sigma_- k} \ dx ds.
\end{equation}  Now, the well-known conservation of energy identity gives a uniform bound on $\int k(t,x) \ dx$ and thus
\begin{eqnarray}
\nonumber \int_{t-\Delta}^t \int_{X_c(s)}^{X_c(s) + C \Delta V_1(t)^{-2}} k(s,x) \ dx ds & \leq & \int_{t-\Delta}^t \int k(s,x) \ dx ds\\
\nonumber & \leq & \int_{t-\Delta}^t C \ ds \\
\label{k} & = &  C \Delta.
\end{eqnarray}
Using (\ref{conefg}), the integral of $\sigma_-$ can be estimated as
\begin{eqnarray*}
\lefteqn{\int_{t-\Delta}^t \int_{X_c(s)}^{X_c(s) + C \Delta V_1(t)^{-2}} \sigma_-(s,x) \ dx ds}\\ & & = \int_{t-\Delta}^t \int_{X(t)-t}^{X(t) -t + C \Delta V_1(t)^{-2}} \sigma_-(s,y+s) \ dy ds\\
& & = \int_{X(t)-t}^{X(t) -t + C \Delta V_1(t)^{-2}} \int_{t-\Delta}^t \sigma_-(s,y+s) \ ds dy\\
& & \leq \int_{X(t)-t}^{X(t) -t + C \Delta V_1(t)^{-2}} C \ dy \\
& & =  C\Delta V_1(t)^{-2}.
\end{eqnarray*}
Combining this with (\ref{k}), the use of Cauchy-Schwarz in (\ref{sqrtks}) gives us
$$ \int_{t-\Delta}^t \int_{X_c(s)}^{X(s)} \int \sum_\alpha f_\alpha \ dv dx dt \leq 3( C \Delta)^{1/2} \left (C \Delta V_1(t)^{-2} \right )^{1/2} \leq C \frac{\Delta}{V_1(t)}$$ and finally, we have
$$V_1(t) \leq V_1(t-\Delta) + C\Delta^{1/2} +C\frac{\Delta}{V_1(t)}.$$
Using (\ref{V1}) this becomes $$ \frac{1}{2} V_1(t) \leq C \left [ \Delta^{1/2} +\frac{\Delta}{V_1(t)} \right ].$$
Hence, $$ C V_1(t)^2 \leq \Delta^{1/2} V_1(t) + \Delta$$ and $$ \left (C + \frac{1}{4} \right ) V_1(t)^2 \leq \left ( \frac{1}{2}V_1(t) + \Delta^{1/2} \right )^2.$$  Finally, $$CV_1(t) = \left ( \sqrt{C + \frac{1}{4}} - \frac{1}{2} \right ) V_1(t) \leq \Delta^{1/2}.$$
By definition, $\Delta \leq t$ and thus $$V_1(t) \leq Ct^{1/2}$$ as long as $V_1(t) \geq 2C_0$.
Similar estimates may be derived if we consider $V_1(t) < -2C_0$.  Combining all cases, we find
$$ \vert V_1(t) \vert \leq \max \{ 2C_0, Ct^{1/2} \}$$ whence $$ \vert V_1(t) \vert \leq 2C_0 + Ct^{1/2}.$$ Since $\alpha$ is arbitrary, this estimate holds for every $f_\alpha$ and taking the supremum over all such characteristics, we finally have $$Q_1(t) \leq C\sqrt{1 + t}$$ completing the proof.
\end{proof}

\section{Behavior of solutions to (\ref{VPN})}
\begin{proof}[Proof of Theorem \ref{Thm6}]
Since we are now considering the classical, neutral system (\ref{VPN}), recall that the characteristic equations are similar to that of (\ref{charVP}), but as for (\ref{RVPN}) they now depend upon the charge of the corresponding species:
\begin{equation}
\label{charVPN} \left. \begin{array}{ccc}
& & \displaystyle \frac{dX_\alpha}{ds} = V_{1\alpha} (s) \\
& & \displaystyle \frac{dV_{1\alpha}}{ds} = e_\alpha E_1(s,X(s)) \\
& & V_{2\alpha} (s,t,x,v) = v_2 \\
\\
& & X_\alpha(t,t,x,v) = x \\
& & V_{1\alpha}(t,t,x,v) = v_1
\end{array} \right \}
\end{equation}
Notice, the only difference between (\ref{charVPN}) and (\ref{charRVPN}) is the absence of the ``hat'' in the equation for $\displaystyle \frac{dX_\alpha}{ds}$.  We begin by deriving a well-known bound which limits the growth of spatial moments for the Vlasov-Poisson system.  Using the Vlasov equation, we find
\begin{eqnarray*}
\frac{d^2}{dt^2} \iint x^2 \sum_\alpha f_\alpha (t,x,v) \ dv dx & = & 2 \frac{d}{dt} \iint xv_1 \sum_\alpha f_\alpha \ dv dx \\
& = & 2 \iint v_1^2 \sum_\alpha f_\alpha \ dv dx\\
& & \ + \ 2 \iint x \sum_\alpha e_\alpha f_\alpha(t,x,v) E_1(t,x) \ dv dx.
\end{eqnarray*}
The first term on the right side of the equation is dominated by the kinetic energy $$\mathcal{K}(t) = \iint \vert v \vert^2 \sum_\alpha f_\alpha (t,x,v) \ dv dx.$$ The second term satisfies
\begin{eqnarray*}
\iint x \sum_\alpha e_\alpha f_\alpha(t,x,v) E_1(t,x) \ dv dx & = & \int \rho x E_1 \ dx\\
& = & \int x \partial_x \left ( \frac{1}{2} E_1^2 \right ) \ dx\\
& = & -\frac{1}{2} \int E_1^2(t,x) \ dx.
\end{eqnarray*}
So, when the two terms are added, the result is uniformly bounded by energy conservation.  Hence, we find
$$ \left \vert \frac{d^2}{dt^2} \iint x^2 \sum_\alpha f_\alpha(t,x,v) \ dv dx \right \vert \leq C$$ and integrating twice in time
$$  \iint x^2 \sum_\alpha f_\alpha(t,x,v) \ dv dx  \leq C(1 + t^2).$$  Using the form of the field, we find for $x > 0$
\begin{eqnarray*}
\vert E_1(t,x) \vert & = & \left \vert \int_x^\infty \rho(t,y) \ dy \right \vert\\
& \leq & C \left (\max_\alpha \vert e_\alpha \vert \right) \int_x^\infty \int \sum_\alpha f_\alpha(t,y,v) \ dv dy\\
& \leq & C \int_x^\infty \int \sum_\alpha f_\alpha(t,y,v) \frac{y^2}{x^2} \ dv dy\\
& \leq & C\frac{1+ t^2}{x^2}.
\end{eqnarray*}
This inequality is similarly obtained for $x < 0$ so that
$$\vert E_1(t,x) \vert \leq C \min \left (1, \frac{1 + t^2}{x^2} \right ).$$

Now, using a method of Horst \cite{Horst} we estimate characteristics.  Let $\alpha$ be given and $(X(t),V(t))$ be characteristics such that $f_\alpha(t,X(t),V(t)) \neq 0$. Suppose that $t > 0$ and $V_1(t) > 0$, and define $$ \tau^* = \inf \{ \tau \in [0,t) : V_1(s) \geq 0 \ \mbox{for} \ s \in[\tau,t] \}.$$
Then, either $V_1(\tau^*)  = 0$ or $\tau^* = 0$ and in both cases, $V_1^2(\tau^*) \leq C$. Using (\ref{charVPN})
\begin{eqnarray*}
V_1^2(t) & \leq & V_1^2(\tau^*) + 2 \left (\max_\alpha \vert e_\alpha \vert \right) \int_{\tau^*}^t \vert E_1(s,X(s)) \vert \ V_1(s) \ ds \\
& \leq & C + C \int_{\tau^*}^t \min \left (1, \frac{1+s^2}{X(s)^2} \right ) \dot{X}(s) \ ds \\
& \leq & C + C \int_{\tau^*}^t \min \left (1, \frac{1+t^2}{X(s)^2} \right ) \dot{X}(s) \ ds \\
& \leq & C + C \int_{X(\tau^*)}^{X(t)} \min \left (1, \frac{1+t^2}{x^2} \right ) \ dx \\
& \leq & C + C \int \min \left (1, \frac{1+t^2}{x^2} \right ) \ dx \\
& \leq & C + C \left ( \int_0^{\sqrt{1+t^2}} 1 \ dx + \int_{\sqrt{1 + t^2}}^\infty \frac{1+t^2}{x^2} \ dx \right ) \\
& \leq & C\left (1 + \sqrt{1 + t^2} \right )\\
& \leq & C(1 + t)
\end{eqnarray*}
Hence, $$V_1(t) \leq C\sqrt{1 + t}.$$
Similar estimates hold for $V_1(t) \leq 0$ and for characteristics of any $f_\alpha$ since $\alpha$ was arbitrary.  With this, we take the supremum over all such characteristics and obtain $Q_1(t) \leq C\sqrt{1 + t}.$
\end{proof}

\medskip

\medskip


\newpage

\begin{thebibliography}{99}


\bibitem{BKR}(MR1750415)
    \newblock J. Batt, M. Kunze and G. Rein,
    \newblock \emph{On the asymptotic behavior of a one-dimensional, monocharged plasma and a rescaling
method},
    \newblock Advances in Differential Equations, \textbf{3} (1998), 271--292.

\bibitem{BFFM}
    \newblock J. R. Burgan, M. R. Feix, E. Fijalkow and A. Munier,
    \newblock \emph{Self-similar and asymptotic solutions for a one-dimensional Vlasov beam},
    \newblock   J. Plasma Physics, \textbf{29} (1983), 139--142.

\bibitem{DD}(MR1104107)
    \newblock L. Desvillettes and J. Dolbeault,
    \newblock \emph{On long time asymptotics of the Vlasov--Poisson--Boltzmann equation},
    \newblock Comm. PDE, \textbf{16} (1991), 451--489.

\bibitem{Dol}(MR1770442)
    \newblock J. Dolbeault,
    \newblock \emph{Time-dependent rescalings and Lyapunov functionals for some kinetic and fluid
    models},
    \newblock  Trans. Theory Stat. Phys., \textbf{29} (2000), 537--549.

\bibitem{DR}(MR1830948)
    \newblock J. Dolbeault and G. Rein,
    \newblock \emph{Time-dependent rescalings and Lyapunov functionals for
    the Vlasov-Poisson and Euler-Poisson systems, and for related models of kinetic
    equations, fluid dynamics and quantum physics},
    \newblock  Math. Methods Appl. Sci., \textbf{11} (2001), 407--432.

\bibitem{Glassey}(MR1379589)
    \newblock R. Glassey,
    \newblock ``The Cauchy Problem in Kinetic Theory,''
    \newblock SIAM, Philadelphia, 1996.

\bibitem{GlaSch}(MR1066384)
    \newblock R. Glassey and J. Schaeffer,
    \newblock \emph{On the ``one and one-half dimensional'' relativistic Vlasov-Maxwell system},
    \newblock Math. Meth. Appl. Sci., \textbf{13} (1990), 169--179.

\bibitem{GPS}(MR2467181)
    \newblock R. Glassey, S. Pankavich and J. Schaeffer,
    \newblock \emph{Decay in time for a one-dimensional, two component plasma},
    \newblock Math. Meth Appl. Sci., \textbf{31} (2008), 2115--2132.

\bibitem{GPSRVM}
    \newblock R. Glassey, S. Pankavich and J. Schaeffer,
    \newblock \emph{Large time behavior of the relativistic Vlasov-Maxwell system in low space dimension},
    \newblock Diff. and Int. Equations, to appear (2009).

\bibitem{GS}(MR0751989)
    \newblock  R. Glassey and W. Strauss,
    \newblock \emph{Remarks on collisionless plasmas},
    \newblock in ``Contemporary Mathematics,'' \textbf{28} (1984), 269--279.

\bibitem{Horst}(MR1032876)
    \newblock  E. Horst,
    \newblock \emph{Symmetric plasmas and their decay},
    \newblock Comm. Math. Phys., \textbf{126} (1990), 613--633.

\bibitem{IR}(MR1414402)
    \newblock R. Illner and G. Rein,
    \newblock \emph{Time decay of the solutions of the Vlasov-Poisson system in the plasma physical case},
    \newblock  Math. Meth. Appl. Sci., \textbf{19} (1996), 1409--1413.

\bibitem{LP}(MR1115549)
    \newblock P. L. Lions and B. Perthame,
    \newblock \emph{Propogation of moments and regularity for the three dimensional Vlasov-Poisson system},
    \newblock Invent. Math., \textbf{105} (1991), 415--430.

\bibitem{Per}(MR1387464)
    \newblock B. Perthame,
    \newblock \emph{Time decay, propagation of low moments and dispersive effects for kinetic equations},
    \newblock Comm. PDE, \textbf{21} (1996), 659--686.

\bibitem{Pfa}(MR1165424)
    \newblock K. Pfaffelmoser,
    \newblock \emph{Global classical solution of the Vlasov-Poisson system in three dimensions for general initial data},
    \newblock J. Diff. Eq., \textbf{95} (1992), 281--303.

\bibitem{Rein}
    \newblock G. Rein,
    \newblock \emph{Collisionless kinetic equations from astrophysics--The Vlasov-Poisson system},
    \newblock in ``Handbook of Differential Equations, Evolutionary Equations,'' Vol. \textbf{
    3},
(eds. C. M. Dafermos and E. Feireisl), Elsevier (2007).

\bibitem{Sch}(MR2293986)
    \newblock J. Schaeffer,
    \newblock \emph{Large-time behavior of a one-dimensional monocharged plasma},
    \newblock Diff. and Int. Eqns., \textbf{20} (2007), 277--292.

\bibitem{Sch3D}(MR1132787)
    \newblock J. Schaeffer,
    \newblock \emph{Global existence of smooth solutions to the Vlasov-Poisson system in three dimensions},
    \newblock Comm. PDE, \textbf{16} (1991), 1313--1335.

\end{thebibliography}
\end{document}